\documentclass[letterpaper]{article}
\usepackage[margin=1in]{geometry}
\usepackage{times} 

\usepackage{amssymb,amsmath,color}
\usepackage{dsfont}
\usepackage{graphicx}
\usepackage{algorithm}
\usepackage{algpseudocode}
\usepackage{amsthm}
\usepackage{amsmath,amssymb,amsthm,amsfonts}
\usepackage{nicefrac}
\usepackage[T1]{fontenc}
\usepackage{natbib} 
\usepackage{hyperref}
\usepackage{dsfont,url}
\usepackage{ulem}
\allowbreak
\allowdisplaybreaks
\newtheorem{thm}{Theorem}[section]

\newtheorem{cor}[thm]{Corollary}
\theoremstyle{definition}
\newtheorem{dfn}{Definition}[section]
\newtheorem{ass}{Assumption}

\theoremstyle{remark}
\newtheorem{rmk}{Remark}[section]
\def\1{\mbox{1\hspace{-.35em}1}} 
\def\argmin{\mbox{argmin}}
\def\R{\mathbb{R}}

\def\N{\mathbb{N}}

\def\E{\mathbb{E}}

\def\e{\,\mbox{e}\,}
\title{High dimensional VAR with low rank transition}
\author{Pierre Alquier$^{(1)}$\footnote{This author gratefully acknowledges financial support from the research programme {\it New Challenges for New Data} from LCL and GENES, hosted by the {\it Fondation du Risque} and from Labex ECODEC (ANR-11-LABEX-0047).}, Karine Bertin$^{(2)}$, Paul Doukhan$^{(3,2)}$\footnote{The work of the second and the third author has been developed within the MME-DII center of excellence (ANR-11-LABEX-0023-01) and with the help of PAI-CONICYT MEC N$^\circ$ 80170072. The authors have been supported by Fondecyt project 1171335 and Mathamsud 18-MATH-07.}, R\'emy Garnier$^{(4,3)}$
\\
\\
\small{{\it (1) RIKEN, Center for Advanced Intelligence Project}, 1-4-1 Nihonbashi Chuo-ku, Tokyo, JAPAN}
\\
\small{{\it (2) CIMFAV, Universidad de Valpara\'iso}. General Cruz 222 -- Valpara\'iso, CHILE}
\\
\small{{\it (3) Laboratoire AGM UMR 8088, Universit\'e Paris Seine}, 2 Bd. Adolphe Chauvin, 95000 Cergy-Pontoise, FRANCE}
\\
\small{{\it (4) CDiscount}, 120 Quai de Bacalan, 33000 Bordeaux, FRANCE}
}

\date{}

\begin{document}
\maketitle

\begin{abstract}
 We propose a vector auto-regressive (VAR) model with a low-rank constraint on the transition matrix. This model is well suited to predict high-dimensional series that are highly correlated, or that are driven by a small number of hidden factors. While our model has formal similarities with factor models, its structure is more a way to reduce the dimension in order to improve the predictions, rather than a way to define interpretable factors. We provide an estimator for the transition matrix in a very general setting, and study its performances in terms of prediction and adaptation to the unknown rank. Our method obtains good result on simulated data, in particular when the rank of the underlying process is small. On macro-economic data from~\cite{giannone2015prior}, our method is competitive with state-of-the-art methods in small dimension, and even improves on them in high dimension.
\end{abstract}

\section{Introduction}

Machine learning became omnipresent in time series forecasting. With the growing ability to collect and store data, one often has to predict series with dimensions from a few hundreds (for economic time series see~\cite{jochmann2013stochastic}) to millions (as it might be the case for sales forecasts for a huge set of items, see for example~\cite{kumar2010using}). The applications fields range from health data analysis (\cite{lipton2016modeling}), electricity consumption forecasting (\cite{gaillard2016additive,Mei17}), economics and GDP forecasting (\cite{cornec2014constructing}), traffic forecasting (\cite{lippi2013short}) and public transport attendance (Chapter 5 in~\cite{carel2019big}), social media analysis (\cite{saha2012learning}), ecology (\cite{purser2009use}), sensor data analysis (\cite{basu2007automatic})... In many of these applications, accurate predictions are of the uttermost importance. For online retail, a bad prediction of the sales before Christmas might lead to inadequate storage and huge losses of money. Poor electricity forecasts might lead to power outage. When it comes to health data, a good prognosis may even be vital\ldots

In many cases, algorithms that were primarily designed for i.i.d. observations are used on time series with good results. For example, the additive model of~\cite{buja1989linear} is used in~\cite{carel2019big} to predict bus lines attendance, with promising results. While generalization bounds were developed first for i.i.d data (see e.g ~\cite{Va98} for an overview), some authors actually proved similar results when the same algorithms are applied on time series. We refer the reader to the pioneering work of~\cite{Mei00} and Steinwart with various co-authors (see~\cite{SC09,SHS09,HS14} for an early overview). Model selection tools are investigated in~\cite{AW12,MSS17} while aggregation methods are studied via PAC-Bayesian bounds in~\cite{alquier2012prediction,ALW13,SK13,LHTG14,AB18}. While these results hold for stationary series, some authors also tackled non-stationary series. De-trending techniques are considered in \cite{CW18} while~\cite{KM15,ADF18} directly prove generalization bounds on non-stationary series. An original approach was also developed in~\cite{giraud2015aggregation} for time-varying AR, relying on sequential prediction methods~ (\cite{cesa2006prediction}).

However, we believe that there is a need to develop machine learning methods that aim at capture stylized facts in time series, in the same spirit that stochastic volatility models were developed to capture stylized facts in finance (see \cite{engle1995arch,francq2019garch}). In this paper, we study a  vector auto-regressive (VAR)  that is suitable to predict high-dimensional series that are strongly correlated, or that are driven by a reasonable number of hidden factors, in the spirit of factor models studied in econometrics~\cite{koop2004forecasting,lam2011estimation,giordani2011bayesian,lam2012factor,hallin2013factor,chan2018invariant}. It is to be noted that factor models naturally lead to low-rank transition matrices. However, in our case, the focus is not on the interpretation of these factors, but on the prediction of the time series. Indeed, even though the series is not driven by a small number of hidden factors, the low-rank constraint on the transition matrix induces a dimension reduction that can improve the predictions, as in~\cite{negahban2011estimation}; see also~\cite{basu2019low} who used such a constraint to detect outliers. We believe that this constraint makes often more sense than sparsity constraints used for example in~\cite{D16}. On the contrary to~\cite{negahban2011estimation}, we focus on the prediction ability of our estimator. We study its prediction performances under a general loss function. While the quadratic loss or the absolute loss are indeed useful, quantile losses can also be used to derive confidence intervals. 

Let us briefly describe the motivation for our models. Assume we deal with an $\mathbb{R}^M$-valued process $(X_t)_{t\geq 0}$ with
\begin{equation}\label{AR1}
    X_t = A X_{t-1} + \xi_t
\end{equation}
for a very large $M$. First, in the hidden factor model, it is assumed that factors $H_t$, that are linear functions of $X_t$, drive the evolution of the process (as a simple example, $X_t$ can be the GDP of each country, and the hidden factors can reflect the general state of the economics at a continent scale). We can write them as $ H_t = V X_t$ where $V$ is $r\times M$ with $r \ll M$. Then, assuming that $X_{t+1}$ can be linearly predicted by $ H_t $, we can predict $X_{t+1}$ by $U H_t = U V X_t $ for some $M\times r$ matrix $U$. At the end of the day, we indeed predict $X_{t+1}$ by $(UV) X_t = A X_t$, but the rank of $A$ is $r\ll M$. However, note that even in situations where we cannot clearly identify hidden factors, a low-rank assumption makes sense. For example, if all the series in $X_t$ are strongly correlated, then $X_t$ can be reasonably well approximated by a rank-one matrix. Hence, to impose a low-rank constraint on the transition matrix (here we assume that the matrix $A$ has a rank $r_0\le M$) will highly reduce the variance, at the price of a small increase of the bias term. This might lead to improved predictions.

Note that the assumption that the coefficient matrix $A$ is low-rank in a multivariate regression model $Y_i = AX_i+\xi_i $ where $Y_i\in\mathbb{R}^s$ and $X_i\in\mathbb{R}^t$ was studied in Econometric theory as early as in the 50's (\cite{An51,Iz75}). It was referred to as RRR (reduced-rank regression). We refer the reader to~\cite{KLT11,BSW11,suzuki2015convergence,ACL17,klopp2017robust,BW17,klopp2017structured,moridomi2018tighter} for state-of-the-art results. Low-rank matrices were actually used to model high-dimensional time series by~\cite{Mei17} and~\cite{alquier2019matrix}, however, the models described in these papers cannot be straightforwardly used for prediction purposes. Here, we study estimation and prediction for the model~\eqref{AR1}.

More precisely, we propose two estimators of the matrix $A$: a fixed-rank estimator and a rank-penalized estimator. The first one is obtained minimizing an empirical prediction error based on a Lipschitz loss function among a specific family of matrix with given rank $r$. In the second one, a procedure based on rank penalization of our empirical error allows us to select a rank $\hat{r}$ and we finally consider the fixed-rank estimator with rank $\hat{r}$. As we mentioned before, we study properties of both estimators in terms of prediction. We prove that the rank-penalized estimator satisfies an oracle inequality adapting to unknown rank. 
We illustrate performances of our estimators on simulated data and a macro-economic dataset from~\cite{giannone2015prior}. We  also compare  our estimators to near low-rank or nuclear estimators proposed in \cite{negahban2011estimation} and \cite{ji2009accelerated}. 

The paper is designed as follows. In the end of the introduction, we provide notations that will be used in the whole paper. We describe our low rank contractive VAR(1) model in Section \ref{s2}; contraction yields the existence of a stationary solution and together yields standard weak dependence properties. The two estimators and their properties are given in Section \ref{s3} whereas the simulation study and application to real data set are  presented in Sections \ref{s4} and \ref{s5}. Finally, the proofs of the results of Section \ref{s3} are postponed to Section \ref{s7}.

\subsection*{Notations}

We introduce here notations that are used in the whole paper. For a vector $x\in\mathbb{R}^M$, $\|x\|$ will denote the Euclidean norm of $x$. We will denote by $|\lambda_1(Q) |\geq \dots \geq |\lambda_M(Q)| \geq 0$ the singular values of an $M\times M$ matrix $Q$, where $r={\rm rank}(Q)$. Note that the singular value decomposition (SVD) of $Q$ can be written
\begin{equation}\label{rankdec} \underbrace{Q}_{M\times M} = \underbrace{U}_{M\times r} \left(\begin{array}{c c c} \lambda_1(Q) & \dots & 0 \\ \vdots & \ddots & \vdots \\ 0 & \dots & \lambda_r(Q) \end{array}\right) \underbrace{V^T}_{r\times M}
\end{equation}
for some semi-unitary matrices $U$ and $V$: $U^T U = V^T V = I_r$, where $I_r$ is the $r\times r$ identity matrix. This means that, from the implicit function theorem, the set of such matrices is a subspace of a Riemanniann manifold with dimension $2(rM-r)+r=r(2M-1)$.\\ For $p\geq 1$ we let $\|Q\|_{S_p}$ denote the Schatten-$p$-norm of $Q$ defined by
$$\|Q\|_{S_p} = \left(\sum_{j=1}^{M}| \lambda_j(Q)|^p\right)^{\frac{1}{p}} $$
if $p<\infty$, and $\|Q\|_{S_\infty}=|\lambda_1(Q)|$.  For $\rho>0$ and $r\in\{1,\ldots,M\}$, we define:
$$ \mathcal{M}(\rho,r) = \left\{Q \in{\cal M}_{M\times M}(\mathbb{R})/\ {\rm rank}(Q) \leq r, \|A\|_{S_\infty} \leq \rho \right\} .$$

\section{Vector auto-regressive process with low-rank constraint}
\label{s2}

We consider the $\mathbb{R}^M$-valued VAR process $(X_t)_{t\geq 1}$, defined from~\eqref{AR1},  where  $(\xi_t)$ are i.i.d and centered random variables and where the matrix $A$ satisfies assumption:

\begin{ass} \label{ass:rho} We assume that $A\in\mathcal{M}(\rho,r_0)$ for some $\rho<1$ and $r_0\in\{1,\dots,M\}$.
\end{ass}
For now, the rank assumption is not restrictive, as the largest possible value $r_0=M$ is not forbidden. However, we will see that a smaller $r_0$ will lead to better rates of convergence in terms of prediction.

The assumption $\rho<1$ is more restrictive, but it is important in what follows. Indeed, setting $F(x,\xi)=A x + \xi$, the following contraction property holds:
$$
\mathbb{E}\|F(x',\xi)-F(x,\xi)\|\le \rho \|x'-x\|,
$$
which yields in case of the existence of a moment for  $\xi$,
$\mathbb{E}\|F(0,\xi)\|=\mathbb{E}\|\xi\|<\infty
$. This implies  the  $\tau$-dependence property (with a geometric decay, $\tau_k\le C\rho^k$) and the ergodicity of this model (see \cite{DDp07}). The assumption $\|A\|_{S_\infty}\le \rho<1$   also implies that the matrix $I-A$ is invertible and its inverse writes as
$$(I-A)^{-1}=I+\sum_{k=1}^\infty A^k.$$
The stationary distribution $\pi$ of this model is the distribution of $(I-A)^{-1}\xi$. From now, we will assume that $X_1\sim\pi$, that is, the process $(X_t)$ is stationary. In the special case of Gaussian noise $\xi_t\sim {\cal N}(0,\Sigma)$ then $$\pi = {\cal N}\left(0,(I-A)^{-1}\Sigma \left((I-A)^{-1}\right)^T\right).$$

We also introduce the following assumption.
\begin{ass}\label{ass:bruit2}
There exist some positive constants $c$ and $d$ such that the noise process $\xi=(\xi^{(1)},\ldots,\xi^{(M)})$ 
satisfies for all $j=1,\ldots, M$
\begin{equation*}
\label{exp}\E \exp\left(\frac{|\xi^{(j)}|}c\right)\le d.
\end{equation*}
\end{ass}
This assumption is for example satisfied for a Gaussian noise $\xi_t \sim \mathcal{N}(0,I_M)$, but it also holds in any case where the components $\xi^{(j)}_t$ of $\xi_t$ are independent and follow any sub-Gaussian distribution, such as a bounded distribution. Assumption~\ref{ass:bruit2} is not required for the process $(X_t)$ to be well defined. However, we will use it to derive our theoretical results on estimation and prediction.

\section{Estimators, prediction and rank selection}\label{s3}

Following the approach described in~\cite{Va98}, we measure the quality of an estimator by its generalization ability, that is, by its out of sample prediction performances. The quality of a prediction will be assessed by a loss function $\ell$: $\ell(u-\tilde{u})$ stands for the cost of having predicted $\tilde{u}$ when the truth appears to be $u$.
\begin{ass}\label{ass:lip} The loss function $\ell$ is $L$-Lipschitz with respect to the Euclidean norm for some $L\leq 1$. That is, for all $x,x'\in \R^M
$, $|\ell(x')-\ell(x)|\le L\|x'-x\|$.
\end{ass}
\noindent
We impose the condition $L=1$ for the sake of simplicity, but note that it is not restrictive as we can rescale the loss anyway. This includes, for example, the Euclidean norm $\ell(x) = \|x\|$ or the max norm $\ell(x) = \max_{1\leq i\leq M}|x_i|$. However, the squared Euclidean norm, $\ell(x)=\|x\|^2$, only satisfies Assumption~\ref{ass:lip} when the process $(X_t)$ and the predictions are bounded.

We are now in position to define the generalization error.
\begin{dfn}
 The generalization error of an $M\times M$ matrix $Q$ is given by
$$ R(Q) = \mathbb{E}\left[\ell\left(X_t - Q  X_{t-1} \right)\right]. $$
\end{dfn}
Note that $R(Q)$ does not depend on $t$ as $(X_t)$ is stationary. Given a sample $X_1,\dots,X_n$ we can obviously estimate this error by its empirical counterpart, and define an estimator based on empirical risk minimization.
\begin{dfn}[{Fixed-rank estimator}]
The empirical error of an $M\times M$ matrix $Q$ is given by
$$ R_n(Q) = \frac{1}{n-1}\sum_{i=2}^n \ell\left(X_i -Q X_{i-1} \right). $$
The rank $r$ empirical risk minimizer (or fixed-rank estimator), $\hat{A}_r$, is defined, for any $r\in\{1,\dots,M\}$, by
 \begin{align}\label{estim}
 R_n(\hat{A}_r) = \min_{Q\in \mathcal{M}(r,\rho)} R_n(Q).
 \end{align}
\end{dfn}
\noindent
Note that in Assumption~\ref{ass:lip}, we only required the loss to be Lipschitz. However, in practice, loss functions that are also convex will be preferred in order to ensure the feasibility of the minimization of the empirical risk.
\begin{rmk}\label{rm:31} Algorithms are known to compute such a low-rank factorization in practice. The most popular method is to write $\hat{A}_r= B C^T$ for $M\times r$ matrices $B$ and $C$, and to optimize with respect to $B$ and $C$ (the constraint $\|\hat{A}_r\|_{S_\infty}\leq \rho$ can be ensured for example by imposing $\|B\|_{S_\infty} \leq \rho $ and $\|C\|_{S_\infty} \leq 1$). The popular ADMM method boils down to the alternate minimization with respect to $B$ and $C$, see Section 9 in~\cite{B11}. This strategy works well in practice, despite its lack of theoretical support in this situation. Even when $\ell$ is convex, the minimization problem is usually non convex with respect to the pair $(B,C)$. Still, recent works~\cite{Ge17} shows that this non-convexity is not ``severe'' in the sense that it is still possible to find a global minimum in a reasonable time. When $\ell$ is not convex, it is not currently known how to compute the exact minimizer in a feasible time.
\end{rmk}
\noindent
The following condition is purely technical and will only be used to make more accurate the statement ``for $n$ large enough''.
\begin{ass}\label{ass:bruit3} We have
$\displaystyle
 n \geq 1 +  16\delta_0^2  Mr \log (9rn) / \mathcal{V}_0
$ where $\mathcal{V}_0=8\e c^2d(2-\rho)/(1-\rho)^3$ and $\delta_0=2 c/(1-\rho)$.
\end{ass}

\begin{thm}
\label{thm-1}
Let  Assumptions \ref{ass:rho}, \ref{ass:bruit2}, \ref{ass:lip} and \ref{ass:bruit3} hold.
 Then for any $\eta>0$, with probability larger than $1-\eta$, we have
\begin{multline}
R(\hat{A}_r) \leq
\min_{Q \in{\cal M}(\rho, r) } R(Q)
+ 4(1+\rho)
\sqrt{
\frac{4\mathcal{V}_0  M^2r \log (9rn) }{n-1}}
\label{ineg}
\\
\nonumber
+ 2(1+\rho) \log\left(\frac{4}{\eta}\right) \sqrt{\frac{\mathcal{V}_0}{4(n-1)
 r \log (9rn)}}
+\frac{\sqrt M}{n}\left(\frac{3 \e dc}{1-\rho}+\frac{\mathcal{V}_0}{2  \delta_0 } \right) +\frac{2\delta_0\sqrt{M}}{n(n-1)}\log\left( \frac{2}{\eta} \right).
\end{multline}
\end{thm}
\begin{cor}Let  Assumptions \ref{ass:rho}, \ref{ass:bruit2}, \ref{ass:lip} and \ref{ass:bruit3} hold.
For any fixed $\eta>0$, we have with probability larger than $(1-\eta)$
$$
R(\hat{A}_r) \leq
\min_{Q \in{\cal M}(\rho, r) } R(Q)
+ D_1 \sqrt{rM^2 \frac{\log n}n},
$$
where $D_1$ denotes a positive constant, only  depending on $c$, $d$, $\rho$ and $\eta$.
\end{cor}
\begin{cor}Let  Assumptions \ref{ass:rho}, \ref{ass:bruit2}, \ref{ass:lip} and \ref{ass:bruit3} hold.
For $\eta=\eta_n=n^{-rM}$, we have for some suitable positive constant $D_2$ depending on $c$, $d$ and $\rho$ that
$$
R(\hat{A}_r) \leq
\min_{Q \in{\cal M}(\rho, r) } R(Q)
+ D_2 \sqrt{rM^2 \frac{\log n}n},
$$
with probability larger that $1-n^{-rM}$.
\end{cor}
\begin{rmk}\label{D+A}[Variations on low rank VAR(1)]
Instead of considering finite rank matrices, an alternative is to consider a low rank perturbation around another matrix.
As an example, the $(D+A)$-model  consists in a diagonal matrix $D$, and $D+A\in{\cal M}(\rho,M)$. Then $D$ is estimated as in the so-called AR case with $M$ decoupled (independent) AR(1)-models in one dimension. The parameter set is now a subset of a Riemanniann manifold with dimension $r(2M-1)+M<2Mr$.
Recall that, in order that $D$ be estimable, the AR(1)-coordinate models must be stable thus their coefficients belong $[-1,1]$. Indeed $D+A\in{\cal M}(\rho,M)$ and $D\in{\cal M}(1,M)$ imply with rank$(A)=r$ that $A\in{\cal M}(1+\rho,r)$.
\end{rmk}
\medskip

Coming back to the low rank model,
note that if the true rank $r_0$ of $A$ is known, then $A\in\mathcal{M}(r_0,\rho)$ leads to
$$
R(\hat{A}_{r_0}) \leq R(A)
+ D_2 \sqrt{r_0 M^2 \frac{\log n}n},
$$
which means that the estimator $\hat{A}_{r_0}$ predicts as well as the true matrix $A$ when $n\rightarrow \infty$ and $r_0$, $M$ satisfy $r_0 M^2 = o(n / \log(n))$. This is true when $M$ and $r_0$ are constants, but also covers high-dimensional asymptotics like $r_0 = \mathcal{O}(1)$ and $M = M_n = \mathcal{O}(n^{1/3})$. However, this estimator requires the knowledge of $r_0$, which is usually not the case. For this reason, we now describe a  model selection procedure for $r$, based on rank penalization.

\noindent
\begin{dfn}[Rank-penalized estimator]\label{def:rank-pen}
 Let  Assumption  \ref{ass:bruit3} holds for at least one value $r\leq M$.  Define $\bar{r}(M,n)$ as the largest integer $r\leq M$ such that $ n \geq 1 +  16 \delta_0^2  r \log (9rn) / \mathcal{V}_0 $.
The rank-penalized estimator  is defined by $ \hat{A} = \hat{A}_{\hat{r}} $ where
$$
\hat{r} = \underset{r=1,\dots, \bar{r}(M,n)}{\argmin} \Biggl[
 R_n(\hat{A}_r) + 2(1+\rho)
\sqrt{
\frac{4\mathcal{V}_0  M^2r \log (9rn) }{n-1}}
 \Biggr].
$$
\end{dfn}
\noindent
In practice, in order to compute $\hat{r}$, we need to know $\rho$, $c$ and $d$ which is still not realistic. However, for such a model selection by penalized risk minimization, it is known anyway that the constants in front of the penalization are usually too large. The slope heuristic leads to better results in practice~\cite{birge2007minimal,BMM12}, we provide more details in Section~\ref{s4}. 
We now provide a theoretical study of the rank-penalized estimator. The following theorem states this estimator predicts almost as well as the best fixed-rank estimator. Of course, which fixed-rank estimator is the best is not known in advance, hence, this estimator is often referred to as the oracle, and the inequality in the theorem as an oracle inequality.
\begin{thm}
\label{thm-2}Let  Assumptions \ref{ass:rho},  \ref{ass:bruit2}, \ref{ass:lip} and \ref{ass:bruit3} hold. Let $\eta>0$.
With probability at least $1-\eta$, we have
\begin{multline*}
 R(\hat{A})
 \leq  \min_{1\leq r\leq \bar{r}(M,n)} \left[ \min_{Q \in {\cal M}(\rho, r) } R(Q) + 4(1+\rho)
\sqrt{
\frac{4\mathcal{V}_0
 M^2r \log (9rn)}{n-1}}  \right]
 \\
 + 3(1+\rho)
\log
\left(\frac{4M}{\eta}\right) \sqrt{ \frac{\mathcal{V}_0}{4(n-1)  \log (9n)}}
+\frac{\sqrt M}{n}\left(\frac{3 \mbox{e} dc}{1-\rho}+\frac{\mathcal{V}_0}{2  \delta_0 } \right) +\frac{2\delta_0\sqrt{M}}{n(n-1)}\log\left( \frac{2}{\eta} \right).
\end{multline*}
\end{thm}\noindent
A more precise analysis of the involved factors yields the following corollaries.
\begin{cor} Under  Assumptions \ref{ass:rho}, \ref{ass:bruit2}, \ref{ass:lip} and \ref{ass:bruit3}, for any fixed $\eta>0$,with probability larger than $1-\eta$
$$
 R(\hat{A})
 \leq  \min_{1\leq r\leq \bar{r}(M,n)} \left[ \min_{Q \in {\cal M}(\rho, r) } R(Q) + 4(1+\rho)
\sqrt{
\frac{4\mathcal{V}_0
 M^2r \log (9rn)}{n-1}}  \right]
+ D_3\sqrt{\frac{(\log M)^2}{n\log n}} +D_4 \frac{\sqrt{M}}{n}.
$$
where $D_3$ and $D_4$ are positive constants depending on $c$, $d$, $\rho$ and $\eta$.
\end{cor}
\begin{cor}Let  Assumptions \ref{ass:rho}, \ref{ass:bruit2}, , \ref{ass:lip} and \ref{ass:bruit3} hold.
For $\eta=\eta_n=n^{-rM}$, we have for some suitable positive constant $D_5$ depending on $c$, $d$ and $\rho$ that
$$
 R(\hat{A})
 \leq  \min_{1\leq r\leq \bar{r}(M,n)} \left[ \min_{Q \in {\cal M}(\rho, r) } R(Q) + 4(1+\rho)
\sqrt{
\frac{4\mathcal{V}_0
 M^2r \log (9rn)}{n-1}}  \right]
 + D_5\log(M)\sqrt{\frac{M^2\log n}{n}}
$$
with probability larger that $1-n^{-M}$.
\end{cor}
When $n$ is large enough so that $r_0\leq \bar{r}(M,n)$, we obtain
$$
 R(\hat{A})
 \leq  R(A) + 4(1+\rho)
\sqrt{
\frac{4\mathcal{V}_0
 M^2r_0 \log (9r_0 n)}{n-1}}
 + D_5\log(M)\sqrt{\frac{M^2\log n}{n}}
$$
so the estimator $\hat{A}$ is adaptive: it does not depend on $r_0$.

\section{Simulations}\label{s4}

\subsection{Simulation design}

In order to evaluate the performance of low-rank reconstruction, we run simulations for different values of the parameters $r_0$ and $n$. We set $M=100$ to have reasonable computation time. The simulations are done with Python and are available online (\cite{Gar19}).

More precisely, for  $(r_0,n)$ fixed,
\begin{itemize}
    \item we generate a matrix of rank $r_0$, as in \eqref{rankdec}, by setting
    $$A=U D V^T$$ where $U$ and $V$ are $M\times r_0$ semi-unitary matrices and $D$ is a diagonal matrix whose diagonal coordinates follow a distribution Beta of parameters $\lambda$ and $1$. The matrices $U$ and $V$ are obtained making orthogonal two $M\times r_0$ matrices whose coefficients are i.i.d uniform on $[0,1]$. Therefore, when $\lambda=1$, the distribution of the singular values of $A$ is uniform. When $\lambda$ is smaller, the singular values tend to be smaller and $A$ is close to some matrix with lower rank.
    \item we generate a sample of length $n$ from an $\mathbb{R}^M$ valued VAR(1)  process $(X_t)_{t\geq 0}$  with
$$ X_t = A X_{t-1} + \xi_t $$
where the $\xi_t$ are truncated centered Gaussian variables with variance $\sigma^2 \mbox{Id}_{M}$, where the i.i.d. coordinates admit the support  $[-10,10]$, and $\sigma=1$. \end{itemize}

\noindent
For each triplet $(r_0,n,\lambda)$, such a data set is simulated $100$ times.

 \subsection{Estimators and quality criteria}

Our estimation procedures use the quadratic loss, that is
$$R_n(Q) = \frac{1}{n-1}\sum_{t=2}^{n} \|X_{t} - Q X_{t-1} \|^2,$$
where we denote by $\|\cdot\|$  the Euclidean norm on $\R^M$. The use of quadratic loss allows one to obtain exact expression for some of the minimization, which speeds up convergence.
Note also that since we use truncated Gaussian Noise,  Assumption \ref{ass:lip} is satisfied.
\medskip

We compare several  estimators in which minimization of this empirical risk (penalized or not) over $\mathcal{M}(\rho,r)$ has to be computed. We use in the simulations $\rho=1$.
As we state in Remark~\ref{rm:31}, we look for minimizers of the form $\hat{Q}= B C^T$ for $M\times r$ matrices $B$ and $C$, and we optimize with respect to $B$ and $C$ by imposing $\|B\|_{S_\infty} \leq 1 $ and $\|C\|_{S_\infty} \leq 1$). We use here the ADMM method that alternates minimization with respect to $B$ and $C$.

\noindent
For each simulation, we compute four estimators:
\begin{itemize}
	\item The standard full-rank VAR(1) estimator magenta(the fixed-rank estimator with rank $M$) for the quadratic loss, i.e.:
	$$\underset{Q \in \mathcal{M}(1,M)}{\argmin}  R_n(Q)$$
	\item The oracle VAR(1) (the fixed rank estimator with rank $r_0$), where we suppose that we know the underlying rank $r_0$. 
	$$ \underset{ Q \in \mathcal{M}(1,r_0)}{\argmin}  R_n(Q) $$
	
	\item The rank-penalized estimator. Note that this estimator can also be rewritten as 
		\begin{equation}\label{eq:toto1} \underset{Q \in \mathcal{M}(1,M)}{\argmin}\left\{ R_n(Q)  + C_1 \sqrt{{\rm rank}(Q)  [C_2 \log (\ {\rm rank}(Q) )] }\right\},
		\end{equation} 
	where $C_1,C_2$  depend on $M$, $n$ and several unknown constants given in Definition \ref{def:rank-pen}. 
	We choose to implement a simpler (and not so different) version of this estimator given by 
	\begin{equation}\label{eq:toto} \hat{Q}_C = \underset{Q \in \mathcal{M}(1,M)}{\argmin}\left\{ R_n(Q)  + C \sqrt{{\rm rank}(Q) }\right\}.
		\end{equation} 
This second estimator have similar performances to \eqref{eq:toto1} for small rank $r_0$ and performs better for higher ranks.
As we explained in Section \ref{s3} the constant in the penalty of the rank-penalized estimator is however too large in practice. A very popular way to fix this is to calibrate this constant $C$ via the slope heuristic. This is what we do here. The slope heuristic can be summarized as follows.
One studies how ${\rm rank}(\hat{Q}_C)$ decreases when $C$ grows. Select as an estimator $\hat{Q}=\hat{Q}_{2 C^*}$ where $C^*$ corresponds to the largest decay in ${\rm rank}(\hat{Q}_C)$. The slope heuristic was introduced by~\cite{birge2007minimal}, more details on its implementation are given in~\cite{BMM12} and its theoretical optimality (for i.i.d data) is studied in~\cite{arlot2009data}. We refer the reader to~\cite{arlot2019minimal} for a recent overview.
	
	\item The near low-rank estimator that we called here \textit{nuclear} as defined in \cite{negahban2011estimation} or \cite{ji2009accelerated}. It introduces an empirical risk minimization penalized by the nuclear norm $\|\cdot\|_*$ . We use the method proposed in \cite{ji2009accelerated} to compute an approximation of this solution. We also use the slope heuristics to compute the penalization constant $C_{\textrm{nuc}}$
$$ \underset{Q \in \mathcal{M}(1,M)}{\argmin}\left\{ R_n(Q)  + C_{\textrm{nuc}} \|Q\|_*  \right\}$$
\end{itemize}  
We want to study the performances of our estimators and to compare them with usual nuclear estimator. As we are interested in studying performances in terms of prediction,  to assess the quality of an estimator $\tilde{A}$, we calculate the excess risk computed on a new sample $X_1^*,\ldots,X_n^*$ generated as in the previous part with the same matrix $A$ used to generate our data:
$$R^*_n(\tilde{A})-R^*_n(A),$$
where
$$R_n^*(\tilde{A}) = \frac{1}{n-1}\sum_{t=2}^{n} \|X_{t}^* - \tilde{A} X_{t-1}^* \|^2.$$

\subsection{Results}

We perform $100$ simulations for different values for the triplets $(r_0,n,\lambda)$. The results are given in the following tables and figures.
Tables~\ref{Table1} and~\ref{Table2} contain the {mean} of the excess risk over the  simulations for the four estimators and different values of  $r_0$, $n$, and $\lambda$. Figures~\ref{Fig1} to~\ref{Fig4} show the dispersion of the excess risk over the $N=100$ simulations for $n \in {1000,500,200}$ and $\lambda = 1$ and for $n=1000$ and $\lambda=0.5$

Moreover  $*$ means that the convergence step of the algorithm was too important.

\begin{table}[htp!]
	\centering
	\begin{tabular}{|l|l|rrrrrrrrrrr|}
		\hline
		& rank $r_0$ &       2   &       3   &       5   &       7   &       10  &       15  &       20  &       30  &       50  &       75  &       100 \\
		\hline
		$\lambda =1$&full rank & 3.27 & 3.21 & 3.21 & 3.23 & 3.20 & 3.15 & 3.27 & 3.26 & 3.25 & 3.24 & 3.28 \\
		&nuclear   & 1.03 & 1.00 & 1.04 & 1.09 & 1.11 & 1.17 & 1.32 & 1.48 & 1.83 & 2.29 & 2.83 \\
		&oracle    & 0.21 & 0.28 & 0.51 & 0.66 & 0.85 & 1.21 & 1.60 & 2.13 & 2.81 & 3.19 & 3.28 \\
		&penalized & 0.14 & 0.17 & 0.28 & 0.54 & 0.70 & 1.39 & 1.40 & 1.78 & 2.34 & 3.10 & 3.55 \\
		\hline
		$\lambda =0.5$&full rank & 3.21 & 3.25 & 3.26 & 3.20 & 3.24 & 3.27 & 3.24 & 3.25 & 3.28 & 3.26 & 3.24 \\
		&nuclear   & 1.28 & 1.31 & 1.35 & 1.31 & 1.37 & 1.43 & 1.47 & 1.55 & 1.76 & 1.99 & 2.22 \\
		&oracle    & 0.21 & 0.32 & 0.53 & 0.68 & 0.97 & 1.35 & 1.70 & 2.19 & 2.89 & 3.21 & 3.24 \\
		&penalized & 0.12 & 0.13 & 0.20 & 0.25 & 0.41 & 0.74 & 1.07 & 1.57 & 1.89 & 2.78 & 3.10 \\
		\hline
	\end{tabular}
	\caption{Mean excess risk  for $n = 1000$ and $\lambda\in\{1, 0.5\}$.}
	\label{Table1}
\end{table}

\begin{table}[htp!]
	\centering
	\begin{tabular}{|l|l|rrrrrrrrrrr|}
		\hline
		& rank $r_0$ &       2   &       3   &       5   &       7   &       10  &       15  &       20  &       30  &       50  &       75  &       100 \\
		\hline
		$n=200$&full rank & 29.59 & 29.34 & 29.91 & 29.54 & 29.59 & 29.74 & 29.62 & 29.52 & 29.67 & 29.49 & 29.68 \\
		&nuclear   &  6.73 &  6.83 &  6.95 &  6.11 &  6.59 &  6.92 &  8.33 &  8.44 &  8.59 & 10.54 & 12.19 \\
		&oracle    &  2.00 &  3.09 &  4.88 &  6.52 &  8.75 & 12.41 & 15.22 & 19.90 & 26.03 & 29.03 & 29.68 \\
		&penalized &  1.04 &  1.09 &  1.27 &  1.43 & 27.05 & 27.55 & 27.83 & 28.66 & 28.62 & 28.75 & 29.64 \\
		\hline
		$n=500$&full rank & 7.30 & 7.26 & 7.26 & 7.27 & 7.39 & 7.24 & 7.25 & 7.23 & 7.39 & 7.37 & 7.19 \\
		&nuclear   & 1.07 & 1.06 & 1.13 & 1.23 & 1.41 & 1.47 & 1.71 & 2.11 & 2.96 & 3.97 & 4.93 \\
		&oracle    & 0.52 & 0.77 & 1.16 & 1.61 & 2.13 & 2.90 & 3.69 & 4.87 & 6.54 & 7.26 & 7.19 \\
		&penalized & 0.29 & 0.31 & 0.44 & 0.71 & 1.12 & 2.19 & 5.77 & 6.74 & 6.85 & 7.05 & 7.33 \\
		\hline
		$n=1000$&full rank & 3.27 & 3.21 & 3.21 & 3.23 & 3.20 & 3.15 & 3.27 & 3.26 & 3.25 & 3.24 & 3.28 \\
		&nuclear   & 1.03 & 1.00 & 1.04 & 1.09 & 1.11 & 1.17 & 1.32 & 1.48 & 1.83 & 2.29 & 2.83 \\
		&oracle    & 0.21 & 0.28 & 0.51 & 0.66 & 0.85 & 1.21 & 1.60 & 2.13 & 2.81 & 3.19 & 3.28 \\
		&penalized & 0.14 & 0.17 & 0.28 & 0.54 & 0.70 & 1.39 & 1.40 & 1.78 & 2.34 & 3.10 & 3.55 \\
		\hline
		$n=2000$&full rank & 1.53 & 1.52 & 1.52 & 1.57 & 1.50 & 1.57 & 1.49 & 1.52 & 1.52 & 1.57 & 1.53 \\
		&nuclear   & 0.73 & 0.73 & 0.73 & 0.80 & 0.76 & 0.84 & 0.83 & 0.92 & 1.07 & 1.29 & 1.42 \\
		&oracle    & 0.08 & 0.12 & 0.20 & 0.30 & 0.38 & 0.56 & 0.69 & 0.95 & 1.29 & 1.54 & 1.53 \\
		&penalized & 0.05 & 0.30 & 0.74 & 1.10 & 0.29 & 0.38 & 1.20 & 1.38 & 1.58 & 1.53 & 2.07 \\
		\hline
		$n=5000$&full rank & 0.60 & 0.57 & 0.60 & 0.61 & 0.61 & 0.62 & 0.59 & 0.61 & 0.58 & 0.57 &  0.62 \\
		&nuclear   & 0.39 & 0.36 & 0.39 & 0.40 & 0.41 & 0.43 & 0.42 & 0.46 & 0.47 & 0.51 & * \\
		&oracle    & 0.04 & 0.05 & 0.07 & 0.10 & 0.14 & 0.20 & 0.25 & 0.38 & 0.48 & 0.55 &  0.62 \\
		&penalized & 0.07 & 0.03 & 0.14 & 0.47 & 0.60 & 0.16 & 0.43 & 1.02 & 1.81 & 1.45 &  1.02 \\
		\hline
	\end{tabular}
	\caption{Mean excess risk  for $\lambda=1$ and $n\in \{200,500,1000,2000,5000\}$}
	\label{Table2}
\end{table}
\noindent

\begin{figure}[htp!] \centering  \includegraphics[scale=0.45]{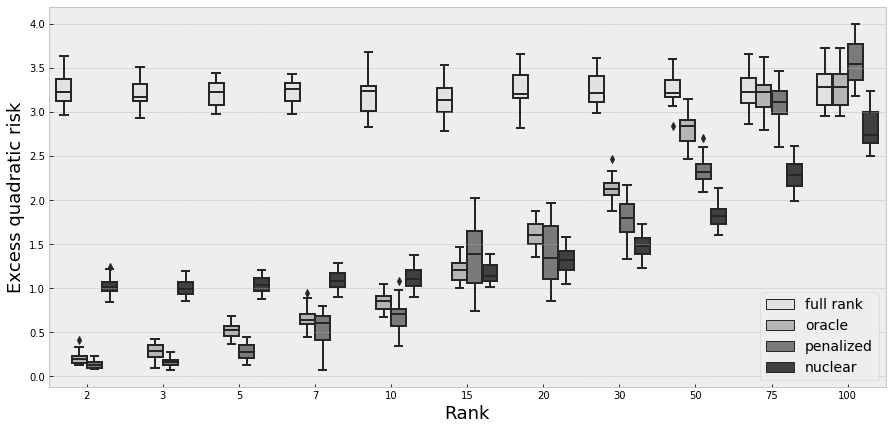}   \caption{Mean excess risk  for $n = 1000$ and $\lambda=1$.}  \label{Fig1}\end{figure}

\begin{figure}[htp!]
\centering
\includegraphics[scale=0.45]{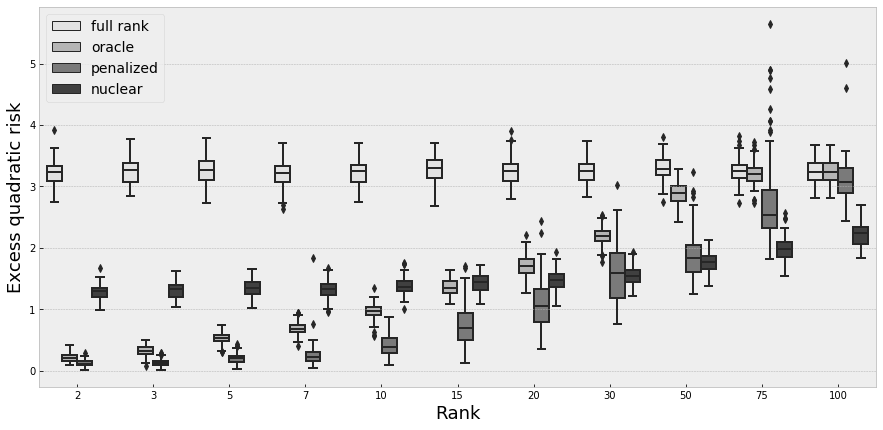}
	\caption{Mean excess risk  for $n = 1000$ and $\lambda=0.5$.}
	\label{Fig2}
\end{figure}


\begin{figure}[htp!]   \centering    \includegraphics[scale=0.45]{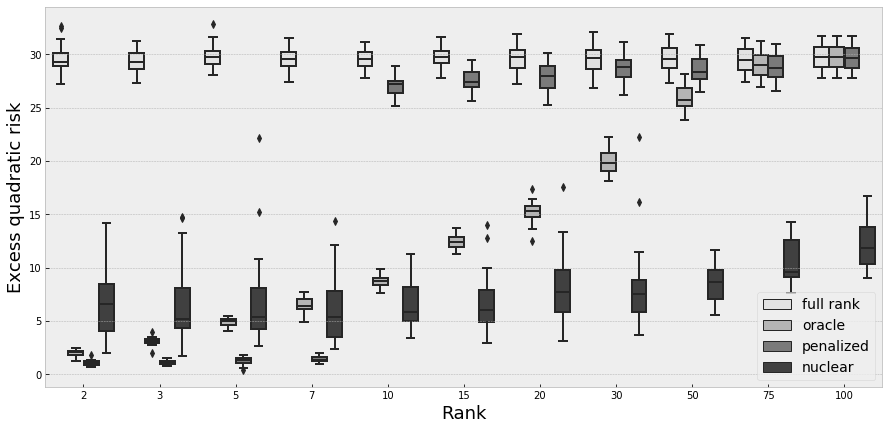}    \caption{Mean excess risk  for $n = 200$ and $\lambda=1$.}   \label{Fig4}\end{figure}

\paragraph{General Remarks}

Performing a full-rank estimation of a matrix of low rank is generally the worst solution. It is normal, because the overfitting is much more important in this case. It doesn't exploit the underlying small rank aspect, so the performance doesn't change with $r_0$.

For all other estimator, the performance degrades with $r_0$. In particular for both penalized and nuclear estimator, it is possible to have worst performance than the full-rank estimator when the underlying rank is closed to M, in particular when the dispersion of singular values is uniform ($\lambda=1$)  . It is also logical, as in this case we are far from the low rank setting, and the penalization becomes a drawback. When $r_0=100$, oracle and full rank estimators are the same.

We also remark that the excess risk decreases when  $n$ increases. When we are really close to the high-dimensional case ($ M \approx n$), a rank threshold appears in the performance of the penalized estimator. Below this rank, the penalized estimator have a good performance. Above this threshold, the penalized estimator has the same performance than the full rank estimator. It suggests that, in this case, the rank chosen by the penalized estimator is close to $M$. 

\paragraph{Comparison with oracle estimator}

When the rank is small, our penalized estimator outperforms the oracle estimator. It is especially the case when we have few observations. In fact, in this case, the penalized estimator choose generally a smaller rank than $r_0$, and it has then less parameters to fit, avoiding generally overfitting. 
However, this effects becomes a drawback when $r_0$ or $n$ increases: then, the complexity of the estimation grows, and therefore the estimation error. 
When the number of small singular values is important ($\lambda = 0.5$), the penalized estimator improves its performance, as the model behaves like a small-rank. 

\paragraph{Comparison with \textit{nuclear} estimator}

Generally the performances of the nuclear estimator are less-rank sensitive than the performance of oracle or penalized estimator. Our hypothesis is that nuclear estimator only performs near low-rank approximation, and that it is harder to be near a "good" very small rank matrix than a "good" larger one. 
This implies that our models outperforms nuclear estimator in small rank settings (or when $\lambda = 0.5$), but it is outperformed when the rank is higher.
The penalized estimator has, relatively, a better performance when $n$ is small. 



\section{Application to a macro-economic data-set}\label{s5}

We want to evaluate the performances of our estimations on real-world datasets. We use macro-economical data known used in \cite{giannone2015prior}. It consists in different US economic indicator. We used the three different datasets:

\begin{itemize}
    \item A small dataset containing $M=3$ quarterly times series between 1959 and 2008.
    \item A medium dataset containing $M=7$ quarterly times series between 1959 and 2008.
    \item A large dataset containing $M=19$ quarterly times series between 1959 and 2006.
\end{itemize}

\noindent
The datasets are described with more details in \cite{giannone2015prior}. We weren't able to get  those whole datasets; hence our results are not easy to compare with those of \cite{giannone2015prior}.
\\
In order to test the models, we estimate them by using the data between 1959 and 2001, and  perform prediction on the period 2001-2006 for GDP (Gross domestic product) and GDP deflator (also called implicit GDP, this is a measure of the level of prices of all new, domestically produced, final goods and services in an economy in a year). We evaluate this prediction using the Mean Squared Error (MSE).

\begin{table}[htp!]
    \centering
    \begin{tabular}{|l|ccc|}
\hline
Method & Small & Medium & Large\\
\hline
Constant Trend& 4.72 & 4.72 & 4.72 \\
Independent AR(1) & 4.56 & 4.56 & 4.56 \\
VAR(1) & 3.58 & 2.70 & 2.90 \\
penalized VAR(1) & \bf{3.57} & 2.78 & 2.84 \\
$(D+A)$-model & 3.87 & \bf{2.58} & \bf{2.83} \\
\hline
\end{tabular}
    \caption{MSE of the GDP forecast.}
    \label{Table3}
\end{table}
\noindent

\begin{table}[htp!]
    \centering
    \begin{tabular}{|l|ccc|}
\hline
Method & Small & Medium & Large\\
\hline
Constant Trend& 1.23 & 1.23 & 1.23 \\
Independent AR(1) & 1.19 & 1.19 & 1.19 \\
VAR(1) & 1.18 & \bf{1.08} & 1.11\\
penalized VAR(1) & 1.18 & 1.08 & \bf{1.09} \\
$(D+A)$-model & \bf{1.17} & 1.10 & 1.15 \\
\hline
\end{tabular}
    \caption{MSE of the GDP deflator forecast.}
    \label{Table4}
\end{table}
\medskip

\noindent
We compare the performances of 5 predictors.
\begin{itemize}
    \item First, we use a naive predictor which asserts  a constant trend at each date.
    \item The second prediction is obtained assuming independent AR(1)-models for each series.
    \item In the third, we use the full rank estimator.
    \item The fourth predictor is the one based on our rank-penalized estimator.
    \item In the fifth predictor, we consider the $(D+A)$-model introduced in the Remark \ref{D+A}.
 \end{itemize}

We sum up the results in Tables \ref{Table3} and \ref{Table4}. The penalized VAR estimator has better results than the Naive or Independent AR estimator on all the datasets. The full-rank VAR algorithm outperforms it for the small or medium datasets. However, if we increase the number of features (or series), the performance of the penalized VAR is less affected than the performance of the full-rank VAR estimator. We suggest that it performs a form of features selection, which reduces over-fitting.

\noindent
The $(D+A)$-model outperforms all other models on GDP forecasts, but the VAR models are better on GDP deflators forecasts. It suggests that the low rank hypothesis may be adapted in other contexts than pure VAR(1) regression.

\section{Conclusion}

We considered a rank-constrained VAR model to predict high-dimensional time series. We proposed a theoretical study of our estimator and proved that it leads to consistent predictions. However, many questions remain opened. Faster rates should be possible for strongly convex losses, in the spirit of~\cite{ALW13}. It would also be necessary to study the performances of the slope heuristic from a theoretical point of view, as was done in the i.i.d setting by~\cite{arlot2009data,arlot2019minimal}.\\
The properties of the $(D+A)$-models introduced in Remark \ref{D+A} are checked in \S\, \ref{s5}; the use of such models deserves a special attention.
Finally other heteroskedastic models may be investigated. E.g.:\begin{itemize}
    \item
  (V)ARCH models may be rewritten from a more elementary rewriting of \cite{BLR06}
$$
X_t=\Sigma_t\xi_t, \qquad
\Sigma_t^T\Sigma_t=a^2\mbox{Id}_M+B( X_{t-1}),$$
for $B(x)=(B_{i,j}(x))_{1\le i,j\le M}$ some symmetric positive matrix valued quad\-ratic form. Recall that the non negative square root of a positive definite matrix exists and is indeed unique. A simple example of this situation is a diagonal quadratic form $$B(x)=\mbox{diag}(b_1(x),\ldots, b_M(x)),$$ with $b_i(x)=x^TQ^{(i)}x$ for $M$ non-negative $M\times M$-matrices $Q^{(1)},\ldots, Q^{(M)}$. \cite{Poignard2019} introduces analogous models with sparsity considerations instead of our suggested of a rank restriction.
The rank condition  turns as the fact that all the matrices
 $Q^{(1)},\ldots, Q^{(M)}$ admit a same decomposition \eqref{rankdec} wrt rank
 $$
 Q^{(i)}=UD^{(i)}V\;;
 $$
\item
(V)INARCH models
$$X_t=P_t(a+AX_{t-1}),$$ for $P_t(\lambda_1,\ldots,\lambda_M)
=(P_t^{(1)}(\lambda_1),\ldots,P_t^{(d)}(\lambda_M))$ defined by i.i.d. arrays of Poisson processes,   some  $a\in\mathbb{R}_+^M$, and some $M\times M$-matrix $A$ with non-negative coefficients.
The low rank model writes here with $A\in{\cal M}(\rho,r)$;
\item
(V)INAR models (defined through thinning operators)
$$
X_t=A\circ X_{t-1}+\xi_t,\quad
(A\circ x)_i=\sum_{j=1}^dA_{i,j}\circ x_j, $$ for some $M\times M$-matrix $A$ with non-negative coefficients, here  $\circ$ stands for the   thinning operator.
The low rank model writes again with $A\in{\cal M}(\rho,r)$.
\end{itemize}
More examples could be inspired from the models detailed in~\cite{D18}. One may think of multiple Volterra processes, bilinear models, switching VAR models, etc\ldots The heteroskedasticity of these models will require new techniques. Indeed, in the proof section below, it will be clear that our results rely on a Bernstein inequality proven in~\cite{DF15} that cannot be used for (V)ARCH models. New concentration tools as, for example, those developed in \cite{DDF} will be used in forthcoming works invoking adapted datasets. \medskip

\noindent The simulations and the real data study demonstrate that such low rank VAR models are meaningful, in particular when the underlying rank of the process is small. Anyway additional features such as heteroskedasticity need to be taken into consideration. Hence, as this was stressed in the introduction, we believe that there is a urgent need for more models involving low rank assumptions for high-dimensional time series.

\section{Proofs}\label{s7}

\subsection{Preliminaries}

Let us start with a few additional notations.
\begin{dfn}
 When Assumption~\ref{ass:bruit2}  is satisfied, we put
$$V_{(n)} =  V\left(  \frac{1}{1-\rho}\Big(K_{n-1}(\rho)\Big)^2 +  \sum_{k=2}^n \Big(K_{n-k}(\rho)\Big)^2  \right),$$
where  $K_n(\rho)=(1-\rho^{n+1})/(1-\rho)$ and $V=8\e c^2d M$ and where $c$ and $d$ are the constants defined in Assumption~\ref{ass:bruit2}.
\end{dfn}
Note that
\begin{equation*}
\frac{V(n)}{n-1}\le \mathcal{V}
\end{equation*}
where $\displaystyle \mathcal{V}=\mathcal{V}_0 M$ and $\delta=\delta_0\sqrt{M}$, and $\mathcal{V}_0=8\e c^2d(2-\rho)/(1-\rho)^3$ and $\delta_0=2 c/(1-\rho)$ are as in Section~\ref{s3} above.

In~\cite{ADF18} the authors propose to use an exponential inequality to control the deviations between the generalization error and the empirical risk in the case of a non-stationary Markov chain. We follow the same approach here. The version of Bernstein inequality from~\cite{DF15} gives directly the following result.
\begin{cor}
 \label{coro-var}
Under Assumptions~\ref{ass:rho} and~\ref{ass:bruit2} we have, for any matrix $Q\in\mathcal{M}(\rho,r)$,
\begin{equation*}
\mathbb{E}\exp\left( \pm s (R_n(Q)-R(Q)) \right)
\leq \exp\left( \frac{\displaystyle s^2(1+\rho)^2  \frac{V_{(n)}}{n-1}}{\displaystyle 2(n-1)-2s(1+\rho)\delta } \right),
\end{equation*}
$\displaystyle \forall s\in\left[0,\frac{n-1}{(1+\rho)\delta}\right)$, with $\displaystyle \delta=\frac{2c\sqrt{M}}{1-\rho}$.
\end{cor}

\begin{proof}[Proof of Corollary~\ref{coro-var}]
This corollary is a consequence of the Bernstein inequality in~\cite{DF15} which requires the following assumption.
\begin{ass}\label{ass:bruit}
There exist some constants $V_1$, $V_2$ and ${C}$ such that
\begin{equation*}
{\mathbb E} \Big[  \Big(  G_{X_1}(X_1)\Big)^k\Big] \leq \frac {k!}{2} V_1 C^{k-2}
\ \text{and} \quad
 {\mathbb E} \Big[  \Big(  G_\xi(\xi)\Big)^k\Big] \leq
 \frac {k!}{2} V_2 C^{k-2} \, .
 \end{equation*}
  \noindent
 where
  $G_{X_1}$ and  $G_\xi$  are  defined by
$$
G_{X_1}(x)= \mathbb{E} \|(I-A)^{-1}\xi-x\|  ,   \qquad
G_{\xi}(y)= \mathbb{E} \|\xi-y \|.
$$
\end{ass}
  \noindent
Note that for $x=(I-A)^{-1}y$.
\begin{equation}
G_{X_1}(x)\le \|(I-A)^{-1}\xi\|_{S_\infty}G_{\xi}(y)\le \frac{1}{1-\rho}G_{\xi}(y).
\end{equation}

 \noindent
Assumption~\ref{ass:bruit2} implies that $\E |\xi^{(1)}|^k\le dc^{k}k!$, for each $k\in\N$.
With the Euclidean norm $\|\xi\|^2=\sum_{i=1}^M(\xi^{(i)})^2$ thus from  H\"older inequality we obtain with $k=2j$
$$\E \|\xi\|^{k}=\sum_{1\le i_1,\ldots,i_j\le M}\E\left( (\xi^{(i_1)})^2\cdots (\xi^{(i_1)})^2\right)
\le\sum_{1\le i_1,\ldots,i_j\le M}\E (\xi^{(1)})^k = \left({\sqrt M}\right)^k\E(\xi^{(1)})^k
\le d\left({\sqrt M}c\right)^k k!
$$
Consider now some odd number such that $k-1=2j$ then
$$
\E \|\xi\|^{k-1}= \|\xi\|^{k-1}_{k-1}\le  \|\xi\|^{k}_{k-1} \le d(c\sqrt M)^{k-1}(k!)^{1-\frac1k}
\le \e d(c\sqrt M)^{k-1}(k-1)!
$$
since from Stirling inequality $k!\ge (k/\e)^k$ and thus $k!^{1-\frac1k}\le \e(k-1)!$; we also use the relation $d\ge1$.
\\
  Note that the independence of coordinates is not required.
\\
Since $\E G^k_\xi(\xi)\le 2^k\E\|\xi\|^k$,
Assumption \ref{ass:bruit} is satisfied with $$C=2c\sqrt M,\quad\  V_2=8\e c^2Md, \quad\mbox{ and}\quad\ V_1=V_2/(1-\rho).$$ Moreover $\delta$ has be chosen as an upper-bound of $CK_n(\rho)$. Then we choose
$\displaystyle\delta=\frac{C}{1-\rho}=\frac{2c\sqrt{M}}{1-\rho}.$
\end{proof}

\subsection{Proof of the theorems of Section~\ref{s3}}

\begin{proof}[Proof of Theorem~\ref{thm-1}] A first issue is to derive some general useful features from empirical processes techniques. Recall that the Frobenius norm of a matrix is $\|Q\|_F=\|Q\|_{S_2}=\sqrt{\mbox{tr }QQ^T}$ satisfies
$$
\|Q\|_{S_\infty}\le \|Q\|_F\le \mbox{rank}(Q)\cdot\|Q\|_{S_\infty}.
$$
Define for of a set of matrices $\cal U$,  its covering number ${\cal N}(\epsilon,{\cal U})$   with respect to the Frobenius norm.\\
In their Lemma 3.1, \cite{CP11} prove a main combinatorial property about the sphere ${\cal S}_r$, which is the set of $M\times M-$matrices $S$ with rank $r$ and with Frobenius norm $\|Q\|_F=1$:
{\it
for $\epsilon > 0$,  ${\cal S}_r$ may be covered by an $\epsilon-$net (in $\|\cdot\|_F-$metric) $\overline{\cal S}_r^\epsilon$ with cardinal $${\cal N}(\epsilon,{\cal S}_r)\le \left(\frac9\epsilon\right)^{(2M+1)r}.$$}
\\
Now denote ${\cal S}_{u,r}$ and ${\cal B}_{u,r}$ respectively the sphere and the ball of such matrices $Q$ with rank $r$ and with either $\|Q\|_F=u$ or $\|Q\|_F\le u$.
\\
An homogeneity argument first allows to derive the existence of an $\epsilon-$net $\overline{\cal S}_{u,r}^\epsilon$ of ${\cal S}_{u,r}$ with cardinal $\le (9u/\epsilon)^{(2M+1)r}$.
\\
Now for $\mu>0$, $\displaystyle{\cal B}_{\mu,r}=\bigcup_{0\le v\le \mu}{\cal S}_{v,r} $ then an $\epsilon-$net $\overline{\cal B}_{\mu,r}^\epsilon$ of ${\cal B}_{\mu,r}$ is provided by $\bigcup_{j=0}^k \overline{\cal S}_{j\epsilon,r}^\epsilon$ with $k=[\mu/\epsilon]+1$. Then
\begin{eqnarray*}{\cal N}(\epsilon,{\cal B}_{\mu,r})&\le& \sum_{j=0}^k {\cal N}(\epsilon,{\cal S}_{j\epsilon,r})\\& \le& (k+1)\left(\frac{9\mu}\epsilon\right)^{(2M+1)r}\\ &\le& 3\cdot\frac \mu\epsilon\left(\frac{9\mu}\epsilon\right)^{(2M+1)r}\!\!\!\!\!\!\!\!\!\!\!\!\!\!\!\!,\qquad\qquad\qquad\qquad \text{ assuming }\quad \frac{\mu}{\epsilon} \geq 1 \\ 
&\le &\left(\frac{9\mu}\epsilon\right)^{(2M+1)r+1}.
\end{eqnarray*}
Still using $\mu/\epsilon\ge 1$ we can simplify further inequalities: $M\ge 1$ so $(2M+1)r+1 \le 3Mr+1$. As in addition $r\leq 1$, one has $3Mr+1\leq 4Mr$, which leads to
\begin{equation*}
{\cal N}(\epsilon, {\cal B}_{\mu,r})\le \left(\frac {9\mu}\epsilon\right)^{4Mr}.
\end{equation*}
The above inequalities relating the various norms imply that ${\cal M}(\rho, r) \subset {\cal B}_{\rho r,r} \subset {\cal B}_{r,r}$ and so
\begin{equation}\label{packing}
{\cal N}(\epsilon, {\cal M}(\rho, r)) \le{\cal N}(\epsilon, {\cal B}_{r,r}) \le \left(\frac {9r}\epsilon\right)^{4Mr}
\end{equation}
as soon as $r/\epsilon \geq 1$.

\medskip

\noindent
Note that the above covering numbers are considered wrt the Frobenius norm, thus the above inequality also concerns the covering numbers wrt $\|\cdot\|_{S_\infty}$.
\medskip
\begin{rmk} For the model $(D+A)$ of Remark \ref{D+A}
quote that  the bound in \eqref{packing}   admits a right hand side inequality transformed as
$$
\le \left(\frac {18r}\epsilon\right)^{4Mr}\left(\frac {2}\epsilon\right)^{M}\le \left(\frac {18r}\epsilon\right)^{4Mr+M}
$$
which does not change the structure of the excess risk behaviour.

\end{rmk}
\noindent
Fix $x_r>0$ and $s_r\in\left[0,\frac{n-1}{(1+\rho)\delta}\right)$ (their value will be given later). By using both Corollary  \ref{coro-var} and the inequality \eqref{packing}, we have
\begin{eqnarray*}
\mbox{Prob}_\epsilon &=& \mathbb{P}\left(\max_{Q\in \overline{\cal B}_{r\rho,r}^\epsilon}|R_n(Q)-R(Q)| > x_r \right)\\
& \leq& \sum_{Q\in \overline{\cal B}_{r\rho,r}^\epsilon} \mathbb{P}\left( |R_n(Q)-R(Q)| > x_r \right)
\\
& \leq& \sum_{Q\in \overline{\cal B}_{\eta,r}^\epsilon} \mathbb{E}\exp\left( s_r |R_n(Q)-R(Q)| - s_r x_r\right)
\\
& \leq& 2 {\cal N}\left(\epsilon, {\cal M}(\rho, r)\right)\exp\left( \frac{s_r^2(1+\rho)^2 \mathcal{V}}{2(n-1)-2s_r(1+\rho)\delta } -s_r x_r \right)
\\
&\le& 2 \left(\frac {9r}\epsilon\right)^{4Mr}\exp\left( \frac{s_r^2(1+\rho)^2  \mathcal{V}}{2(n-1)-2s_r(1+\rho)\delta } -s_r x_r \right).
\end{eqnarray*}
Now, for any $Q\in{\cal M}(\rho, r)$ define $Q^{\epsilon}$ by
$$ Q^\epsilon = \underset{B\in\overline{\cal B}_{r\rho,r}^\epsilon}{\operatorname{argmin}}  \|Q - B \| . $$
Obviously
$$
\left| \ell(X_t-Q^\epsilon  X_{t-1}) -  \ell(X_t -Q X_{t-1} )\right|
\leq  \|(Q^\epsilon-Q)X_{t-1})\|
\leq  \|Q^\epsilon-Q\|_{S_\infty}\|X_{t-1}\|\leq \epsilon\|X_{t-1}\|,
$$
and as a consequence,
$$|R_n(Q^\epsilon)-R_n(Q)| \leq   \epsilon\cdot\frac{\sum_{t=1}^{n-1}\|X_t\|}{n-1},$$
and
$$|R(Q^\epsilon)-R(Q)| \leq   \epsilon\cdot\mathbb{E}\|X_1\|\le \frac{\epsilon}{1-\rho}\cdot\mathbb{E}\|\xi_1\|.$$
Applying Bernstein Inequality (3.3) in Proposition 3.1 of \cite{DF15} with $f(X_1,\dots,X_{n-1})=\sum_{j=1}^{n-1} \|X_j\|$ and $t\delta=\frac12$ we have, for any $y>0$,
\begin{align*}
& \mathbb{P}\left(\sum_{t=1}^{n-1}\|X_t\| > (n-1)\mathbb{E}\|X_1\| + y \right) \\
& \leq \mathbb{E}\exp\left[\frac{1}{2\delta}\left(\sum_{t=1}^{n-1}\|X_t\| - (n-1)\mathbb{E}\|X_1\| - y\right) \right]
\\
& \leq \exp\left(\frac{\left(\frac{1}{2\delta}\right)^2 V_{(n)}}{2\left(1-\frac{1}{2}\right)} - \frac{ y}{2\delta} \right) = \exp\left(\frac{ V_{(n)} }{4\delta^2}-\frac{ y }{2\delta} \right).
\end{align*}

\noindent
Now let us consider the ``favorable'' event
\begin{equation*}
\mathcal{E}_r= \left\{ \sum_{t=1}^{n-1}\|X_t\| \leq (n-1)\mathbb{E}\|X_1\| + y \right\}
  \bigcap \left\{\max_{Q\in \overline{\cal B}_{r\rho,r}^\epsilon}|R_n(Q)-R(Q)| \leq x_r  \right\}.
\end{equation*}
The previous inequalities show that
\begin{equation}
\label{proba-e}
\mathbb{P}\left( \mathcal{E}_r^c\right)\leq  \exp\left(\frac{ V_{(n)}}{4\delta^2}-\frac{ y }{2\delta} \right)
+ 2 \left(\frac {9r}\epsilon\right)^{4Mr}\exp\left( \frac{s_r^2(1+\rho)^2  \mathcal{V}}{2(n-1)-2s_r(1+\rho)\delta } -s_r x_r \right).
\end{equation}
\noindent
On $\mathcal{E}_r$, we have:
\begin{align*}
R(\hat{A}_{r})
& \leq R(\hat{A}_{r}^{\epsilon}) + \frac{\epsilon}{1-\rho}\cdot\mathbb{E}\|\xi_1\|
\\
& \leq R_n(\hat{A}_{r}^{\epsilon}) + x_r  +\frac{\epsilon}{1-\rho}\cdot\mathbb{E}\|\xi_1\|
\\
& \leq R_n(\hat{A}_{r}) + x_r + \epsilon  \left[ 2\cdot \frac{\mathbb{E}\|\xi_1\|}{1-\rho} + \frac{y}{n-1}\right]
\\
& \le \min_{Q  \in \overline{\cal B}_{r\rho,r}^\epsilon } R_n(Q) + x_r + \epsilon  \left[ 2 \cdot\frac{\mathbb{E}\|\xi_1\|}{1-\rho} + \frac{y}{n-1}\right]
\\
& \leq \min_{Q \in\overline{\cal B}_{r\rho,r}^\epsilon } R(Q) + 2x_r + \epsilon  \left[ 2\cdot \frac{\mathbb{E}\|\xi_1\|}{1-\rho} + \frac{y}{n-1}\right]
\\
& \leq \min_{Q \in {\cal M}(\rho, r) } R(Q) + 2x_r + \epsilon  \left[ 3\cdot \frac{\mathbb{E}\|\xi_1\|}{1-\rho} + \frac{y}{n-1}\right] .
\end{align*}
In particular, the choice $\epsilon = 1/n$ ensures:
\begin{equation}
\label{almost-done}
R(\hat{A}_{r}) \leq \min_{Q \in {\cal M}(\rho, r)} R(Q) + 2x_r + \frac{1}{n} \left[ 3\cdot \frac{\mathbb{E}\|\xi_1\|}{1-\rho} + \frac{y}{n-1}\right].
\end{equation}
Note that this choice is allowed: indeed, the only condition on $\epsilon$ is $r/\epsilon \geq 1$ and $r/\epsilon = rn \geq 1$.
Remind that $$\log \mathcal{N}(\epsilon,{\cal M}(\rho, r))\le 4Mr \log (9rn). $$ \\ Fix $\eta>0$ and put:
$$ x_r = \frac{s_r(1+\rho)^2  \mathcal{V}}{2(n-1)-2s_r(1+\rho)\delta } + \frac{ 4Mr \log (9rn) + \log\left(\frac{4}{\eta}\right)}{s_r} $$
and
$$ y =  2\delta\log\left(\frac{2}{\eta}\right) + \frac{ V_{(n)}}{2\delta }.$$
Note that, plugged into~\eqref{proba-e}, these choices ensure $\mathbb{P}(\mathcal{E}^c)\leq \eta/2+\eta/2=\eta$.\\
Put
$$s_r= \frac{1}{(1+\rho)}\sqrt{\frac{4(n-1) Mr \log (9rn)}{\mathcal{V}} }.$$
As soon as $ 2s_r(1+\rho)\delta  \leq n-1$, that is actually ensured by the condition $ n \geq 1 +  16\delta^2  Mr \log (9rn) / \mathcal{V} $,
we have:
\begin{align*}
x_r & \leq
\frac{s_r(1+\rho)^2 \mathcal{V}}{n-1} +
\frac{ 4Mr \log (9rn) + \log\left(\frac{4}{\eta}\right)}{s_r}
\\
& = 2(1+\rho)
\sqrt{
\frac{4\mathcal{V}
 Mr \log (9rn)}{n-1}}
 \\
& \quad + (1+\rho)
\log
\left(\frac{4}{\eta}\right) \sqrt{ \frac{\mathcal{V}}{(n-1) 4Mr \log (9rn)}}.
\end{align*}
Plugging the expressions of $x_r$ and $y$  into~\eqref{almost-done} gives:
\begin{multline}
R(\hat{A}_r) \leq
\min_{Q \in{\cal M}(\rho, r) } R(Q)
+ 4(1+\rho)
\sqrt{
\frac{4\mathcal{V}  Mr \log (9rn) }{n-1}}
\label{ineg}
\\
\nonumber
+ 2(1+\rho) \log\left(\frac{4}{\eta}\right) \sqrt{\frac{\mathcal{V}}{4(n-1)
 Mr \log (9rn)}}
+\frac{1}{n}\cdot \frac{3\mathbb{E}\| \xi_1 \|}{1-\rho} +\frac{2\delta}{n(n-1)}\log\left( \frac{2}{\eta} \right) +\frac{\mathcal{V}}{2 n \delta }
\nonumber
\end{multline}
which ends the proof taking into account that $\mathcal{V}=\mathcal{V}_0 M$ and $\delta=\delta_0\sqrt{M}$.
\end{proof}

\begin{proof}[Proof of Theorem~\ref{thm-2}]
For short, put
$${\rm pen}(r) = 2(1+\rho)
\sqrt{
\frac{4\mathcal{V}_0  M^2r \log (9rn) }{n-1}} $$
so that
$$\hat{r} = \underset{r=1,\dots, \bar{r}(M,n)}{\operatorname{argmin}} \Biggl[
 R_n(\hat{A}_r) + {\rm pen}(r) \Biggr].$$
For short, put $\bar{r}=\bar{r}(M,n)$. We keep the notations of the previous proof, we will just change the values of $x_r$. From the classical union bound argument we derive:
$$
\mathbb{P}\left( \bigcup_{r=1}^{r_0} \mathcal{E}_r^c\right)\leq  \exp\left(\frac{ V_{(n)}}{4\delta^2}-\frac{ y }{2\delta} \right)
+ 2 \sum_{r=1}^{r_0} \left(\frac {9r}\epsilon\right)^{4Mr}\exp\left( \frac{s_r^2(1+\rho)^2  \mathcal{V}}{2(n-1)-2s(1+\rho)\delta } -s_r x_r \right).
$$
We take $y$, $s_r$ as in the previous proof and
$$ x_r = \frac{s_r(1+\rho)^2  \mathcal{V}}{2(n-1)-2s_r(1+\rho)\delta } + \frac{ 4Mr \log (9rn) + \log\left(\frac{4 M}{\eta}\right)}{s_r} $$
which implies
$$
x_r \leq 2(1+\rho)
\sqrt{
\frac{4\mathcal{V}
 Mr \log (9rn)}{n-1}} + (1+\rho)
\log
\left(\frac{4M}{\eta}\right) \sqrt{ \frac{4\mathcal{V}}{(n-1) Mr \log (9rn)}}.
$$
This leads to
$$ \mathbb{P}\left( \bigcup_{r=1}^{r_0} \mathcal{E}_r^c\right) \leq \frac{\eta}{2} + 2\sum_{r=1}^{r_0} \frac{\eta}{4M} =  \frac{\eta}{2} + \frac{r_0 \eta}{2M} \leq \eta.$$
Also note that for any $r\in\{1,\dots,M\}$,
\begin{align*}
x_r-{\rm pen}(r)
& =  (1+\rho)
\log
\left(\frac{4M}{\eta}\right) \sqrt{ \frac{\mathcal{V}}{4(n-1) Mr \log (9rn)}}
\\
& \leq (1+\rho)
\log
\left(\frac{4M}{\eta}\right) \sqrt{ \frac{\mathcal{V}}{4(n-1) M \log (9n)}}.
\end{align*}
Now, on the favorable event $\bigcap_{r=1}^{\bar{r}}\mathcal{E}_r$ we have
\begin{align*}
R(\hat{A})
& = R(\hat{A}_{\hat{r}})
\\
& \leq R(\hat{A}_{\hat{r}}^{\epsilon}) + \frac{\epsilon}{1-\rho}\cdot\mathbb{E}\|\xi_1\| \\
& \leq R_n(\hat{A}_{\hat{r}}^{\epsilon}) + x_{\hat{r}}  +\frac{\epsilon}{1-\rho}\cdot\mathbb{E}\|\xi_1\|
\\
& \leq R_n(\hat{A}_{r}) + x_{\hat{r}} + \epsilon  \left[ 2\cdot \frac{\mathbb{E}\|\xi_1\|}{1-\rho} + \frac{y}{n-1}\right]
\\
& \leq R_n(\hat{A}_{r}) +{\rm pen}(\hat{r})+ x_{\hat{r}}-{\rm pen}(\hat{r}) + \epsilon  \left[ 2\cdot \frac{\mathbb{E}\|\xi_1\|}{1-\rho} + \frac{y}{n-1}\right]
\\
& = \min_{1\leq r\leq \bar{r}} \left[ R_n(\hat{A}_{r}) + x_{r}\right] + [x_{\hat{r}}-{\rm pen}(\hat{r})] + \epsilon  \left[ 2\cdot \frac{\mathbb{E}\|\xi_1\|}{1-\rho} + \frac{y}{n-1}\right]
\\
& \le \min_{1\leq r\leq \bar{r}} \left[ \min_{Q  \in \overline{\cal B}_{r\rho,r}^\epsilon } R_n(Q) + x_r \right] + [x_{\hat{r}}-{\rm pen}(\hat{r})]  + \epsilon  \left[ 2 \cdot\frac{\mathbb{E}\|\xi_1\|}{1-\rho} + \frac{y}{n-1}\right]
\\
& \le \min_{1\leq r\leq \bar{r}} \left[ \min_{Q  \in \overline{\cal B}_{r\rho,r}^\epsilon } R(Q) + 2 x_r \right] + [x_{\hat{r}}-{\rm pen}(\hat{r})]  + \epsilon  \left[ 2 \cdot\frac{\mathbb{E}\|\xi_1\|}{1-\rho} + \frac{y}{n-1}\right]
\\
& \leq \min_{1\leq r\leq \bar{r}} \left[ \min_{Q \in {\cal M}(\rho, r) } R(Q) + 2x_r \right]
\\
& \quad \quad \quad +(1+\rho)
\log
\left(\frac{4M}{\eta}\right) \sqrt{ \frac{\mathcal{V}}{4(n-1) M \log (9n)}}  + \epsilon  \left[ 3\cdot \frac{\mathbb{E}\|\xi_1\|}{1-\rho} + \frac{y}{n-1}\right] .
\end{align*}
Replace $y$ and $x_r$ by their value to obtain
\begin{multline*}
 R(\hat{a})
 \leq  \min_{1\leq r\leq \bar{r}} \left[ \min_{Q \in {\cal M}(\rho, r) } R(Q) + 4(1+\rho)
\sqrt{
\frac{4\mathcal{V}
 Mr \log (9rn)}{n-1}}  \right]
 \\
 + 3(1+\rho)
\log
\left(\frac{4M}{\eta}\right) \sqrt{ \frac{\mathcal{V}}{4(n-1) M \log (9n)}}
+\frac{1}{n}\cdot \frac{3\mathbb{E}\| \xi_1 \|}{1-\rho} +\frac{2\delta}{n(n-1)}\log\left( \frac{2}{\eta} \right) +\frac{\mathcal{V}}{2 n \delta }.
\end{multline*}
This ends the proof.
\end{proof}

\bibliographystyle{apalike}

\begin{thebibliography}{}

\bibitem[Alquier et~al., 2019a]{ACL17}
Alquier, P., Cottet, V., and Lecu{\'e}, G. (2019a).
\newblock Estimation bounds and sharp oracle inequalities of regularized
  procedures with lipschitz loss functions.
\newblock {\em The Annals of Statistics}, 47(4):2117--2144.

\bibitem[Alquier et~al., 2019b]{ADF18}
Alquier, P., Doukhan, P., and Fan, X. (2019b).
\newblock Exponential inequalities for nonstationary {M}arkov chains.
\newblock {\em Dependence Modeling}, (7):150--168.

\bibitem[Alquier and Guedj, 2018]{AB18}
Alquier, P. and Guedj, B. (2018).
\newblock Simpler {PAC-B}ayesian bounds for hostile data.
\newblock {\em Machine Learning}, 107(5):887--902.

\bibitem[Alquier and Li, 2012]{alquier2012prediction}
Alquier, P. and Li, X. (2012).
\newblock Prediction of quantiles by statistical learning and application to
  gdp forecasting.
\newblock In {\em International Conference on Discovery Science}, pages 22--36.
  Springer.

\bibitem[Alquier et~al., 2013]{ALW13}
Alquier, P., Li, X., and Wintenberger, O. (2013).
\newblock Prediction of time series by statistical learning: general losses and
  fast rates.
\newblock {\em Dependence Modeling}, 1:65--93.

\bibitem[Alquier and Marie, 2019]{alquier2019matrix}
Alquier, P. and Marie, N. (2019).
\newblock Matrix factorization for multivariate time series analysis.
\newblock {\em Electronic journal of statistics}, 13(2):4346--4366.

\bibitem[Alquier and Wintenberger, 2012]{AW12}
Alquier, P. and Wintenberger, O. (2012).
\newblock Model selection for weakly dependent time series forecasting.
\newblock {\em Bernoulli}, 18(3):883--913.

\bibitem[Anderson, 1951]{An51}
Anderson, T.~W. (1951).
\newblock Estimating linear restrictions on regression coefficients for
  multivariate normal distributions.
\newblock {\em The Annals of Mathematical Statistics}, 22(3):327--351.

\bibitem[Arlot, 2019]{arlot2019minimal}
Arlot, S. (2019).
\newblock Minimal penalties and the slope heuristics: a survey.
\newblock {\em arXiv preprint arXiv:1901.07277}.

\bibitem[Arlot and Massart, 2009]{arlot2009data}
Arlot, S. and Massart, P. (2009).
\newblock Data-driven calibration of penalties for least-squares regression.
\newblock {\em Journal of Machine learning research}, 10(Feb):245--279.

\bibitem[Basu et~al., 2019]{basu2019low}
Basu, S., Li, X., and Michailidis, G. (2019).
\newblock Low rank and structured modeling of high-dimensional vector
  autoregressions.
\newblock {\em IEEE Transactions on Signal Processing}, 67(5):1207--1222.

\bibitem[Basu and Meckesheimer, 2007]{basu2007automatic}
Basu, S. and Meckesheimer, M. (2007).
\newblock Automatic outlier detection for time series: an application to sensor
  data.
\newblock {\em Knowledge and Information Systems}, 11(2):137--154.

\bibitem[Baudry et~al., 2012]{BMM12}
Baudry, J.-P., Maugis, C., and Michel, B. (2012).
\newblock Slope heuristics: overview and implementation.
\newblock {\em Statistics and Computing}, 22(2):455--470.

\bibitem[Bauwens et~al., 2006]{BLR06}
Bauwens, L., Laurent, S., and V.~K.~Rombouts, J. (2006).
\newblock Multivariate {GARCH} models: A survey.
\newblock {\em Journal of Applied Econometrics}, 21:79--109.

\bibitem[Bing and Wegkamp, 2019]{BW17}
Bing, X. and Wegkamp, M.~H. (2019).
\newblock Adaptive estimation of the rank of the coefficient matrix in
  high-dimensional multivariate response regression models.
\newblock {\em The Annals of Statistics}, 47(6):3157--3184.

\bibitem[Birg{\'e} and Massart, 2007]{birge2007minimal}
Birg{\'e}, L. and Massart, P. (2007).
\newblock Minimal penalties for gaussian model selection.
\newblock {\em Probability theory and related fields}, 138(1-2):33--73.

\bibitem[Boyd et~al., 2011]{B11}
Boyd, S., Parikh, N., Chu, E., Peleato, B., Eckstein, J., et~al. (2011).
\newblock Distributed optimization and statistical learning via the alternating
  direction method of multipliers.
\newblock {\em Foundations and Trends{\textregistered} in Machine learning},
  3(1):1--122.

\bibitem[Buja et~al., 1989]{buja1989linear}
Buja, A., Hastie, T., and Tibshirani, R. (1989).
\newblock Linear smoothers and additive models.
\newblock {\em The Annals of Statistics}, 17(2):453--510.

\bibitem[Bunea et~al., 2011]{BSW11}
Bunea, F., She, Y., and Wegkamp, M.~H. (2011).
\newblock Optimal selection of reduced rank estimators of high-dimensional
  matrices.
\newblock {\em The Annals of Statistics}, 39(2):1282--1309.

\bibitem[Cand\`es and Plan, 2011]{CP11}
Cand\`es, E.~J. and Plan, Y. (2011).
\newblock Tight oracle inequalities for low-rank matrix recovery from a minimal
  number of noisy random measurements.
\newblock {\em IEEE Trans. Inform. Theory}, 57(4):2342--2359.

\bibitem[Carel, 2019]{carel2019big}
Carel, L. (2019).
\newblock {\em Big data analysis in the field of transportation}.
\newblock PhD thesis, PhD Thesis, Universit{\'e} Paris-Saclay.

\bibitem[Cesa-Bianchi and Lugosi, 2006]{cesa2006prediction}
Cesa-Bianchi, N. and Lugosi, G. (2006).
\newblock {\em Prediction, learning, and games}.
\newblock Cambridge university press.

\bibitem[Chan et~al., 2018]{chan2018invariant}
Chan, J., Leon-Gonzalez, R., and Strachan, R.~W. (2018).
\newblock Invariant inference and efficient computation in the static factor
  model.
\newblock {\em Journal of the American Statistical Association},
  113(522):819--828.

\bibitem[Chen and Wu, 2018]{CW18}
Chen, L. and Wu, W.~B. (2018).
\newblock Testing for trends in high-dimensional time series.
\newblock {\em Journal of the American Statistical Association}.

\bibitem[Cornec, 2014]{cornec2014constructing}
Cornec, M. (2014).
\newblock Constructing a conditional {GDP} fan chart with an application to
  {F}rench business survey data.
\newblock {\em {OECD} Journal: Journal of Business Cycle Measurement and
  Analysis}, 2013(2):109--127.

\bibitem[Davis et~al., 2016]{D16}
Davis, R.~A., Zang, P., and Zheng, T. (2016).
\newblock Sparse vector autoregressive modeling.
\newblock {\em Journal of Computational and Graphical Statistics},
  25(4):1077--1096.

\bibitem[De~Castro et~al., 2017]{Mei17}
De~Castro, Y., Goude, Y., H{\'e}brail, G., and Mei, J. (2017).
\newblock Recovering multiple nonnegative time series from a few temporal
  aggregates.
\newblock In {\em ICML 2017-34th International Conference on Machine Learning},
  pages 1--9.

\bibitem[Dedecker et~al., 2019]{DDF}
Dedecker, J., Doukhan, P., and Fan, X. (2019).
\newblock Deviation inequalities for separately {L}ipschitz functionals of
  composition of random functions.
\newblock {\em Journal of Mathematical Analysis and its Applications},
  479(2):1549--1568.

\bibitem[Dedecker et~al., 2007]{DDp07}
Dedecker, J., Doukhan, P., Lang, G., Rafael, L.~R., Louhichi, S., and Prieur,
  C. (2007).
\newblock Weak dependence.
\newblock In {\em Weak dependence: With examples and applications}, pages
  9--20. Springer.

\bibitem[Dedecker and Fan, 2015]{DF15}
Dedecker, J. and Fan, X. (2015).
\newblock Deviation inequalities for separately {L}ipschitz functionals of
  iterated random functions.
\newblock {\em Stochastic Processes and their Applications}, 125(1):60--90.

\bibitem[Doukhan, 2018]{D18}
Doukhan, P. (2018).
\newblock {\em Stochastic models for time series}, volume~80.
\newblock Series Math\'ematiques et Applications, Springer.

\bibitem[Engle, 1995]{engle1995arch}
Engle, R. (1995).
\newblock {\em ARCH: selected readings}.
\newblock Oxford University Press.

\bibitem[Francq and Zakoian, 2019]{francq2019garch}
Francq, C. and Zakoian, J.-M. (2019).
\newblock {\em {GARCH} models: structure, statistical inference and financial
  applications}.
\newblock Wiley.

\bibitem[Gaillard et~al., 2016]{gaillard2016additive}
Gaillard, P., Goude, Y., and Nedellec, R. (2016).
\newblock Additive models and robust aggregation for {GEFCom2014} probabilistic
  electric load and electricity price forecasting.
\newblock {\em International Journal of forecasting}, 32(3):1038--1050.

\bibitem[Garnier, 2019]{Gar19}
Garnier, R. (2019).
\newblock Simulations for penalized estimation of {VAR} with law rank.
\newblock {\em https://github.com/garnier94/Simulation\_low\_rank\_VAR1}.

\bibitem[Ge et~al., 2017]{Ge17}
Ge, R., Jin, C., and Zheng, Y. (2017).
\newblock No spurious local minima in nonconvex low rank problems: A unified
  geometric analysis.
\newblock In {\em Proceedings of the 34th International Conference on Machine
  Learning-Volume 70}, pages 1233--1242. JMLR. org.

\bibitem[Giannone et~al., 2015]{giannone2015prior}
Giannone, D., Lenza, M., and Primiceri, G.~E. (2015).
\newblock Prior selection for vector autoregressions.
\newblock {\em Review of Economics and Statistics}, 97(2):436--451.

\bibitem[Giordani et~al., 2011]{giordani2011bayesian}
Giordani, P., Pitt, M., and Kohn, R. (2011).
\newblock Bayesian inference for time series state space models.
\newblock In {\em The Oxford Handbook of Bayesian Econometrics}.

\bibitem[Giraud et~al., 2015]{giraud2015aggregation}
Giraud, C., Roueff, F., and Sanchez-Perez, A. (2015).
\newblock Aggregation of predictors for nonstationary sub-linear processes and
  online adaptive forecasting of time varying autoregressive processes.
\newblock {\em The Annals of Statistics}, 43(6):2412--2450.

\bibitem[Hallin and Lippi, 2013]{hallin2013factor}
Hallin, M. and Lippi, M. (2013).
\newblock Factor models in high-dimensional time series—a time-domain
  approach.
\newblock {\em Stochastic processes and their applications}, 123(7):2678--2695.

\bibitem[Hang and Steinwart, 2014]{HS14}
Hang, H. and Steinwart, I. (2014).
\newblock Fast learning from $\alpha$-mixing observations.
\newblock {\em Journal of Multivariate Analysis}, 127:184--199.

\bibitem[Izenman, 1975]{Iz75}
Izenman, A.~J. (1975).
\newblock Reduced-rank regression for the multivariate linear model.
\newblock {\em Journal of multivariate analysis}, 5(2):248--264.

\bibitem[Ji and Ye, 2009]{ji2009accelerated}
Ji, S. and Ye, J. (2009).
\newblock An accelerated gradient method for trace norm minimization.
\newblock In {\em Proceedings of the 26th annual international conference on
  machine learning}, pages 457--464. ACM.

\bibitem[Jochmann et~al., 2013]{jochmann2013stochastic}
Jochmann, M., Koop, G., Leon-Gonzalez, R., and Strachan, R.~W. (2013).
\newblock Stochastic search variable selection in vector error correction
  models with an application to a model of the {UK} macroeconomy.
\newblock {\em Journal of Applied Econometrics}, 28(1):62--81.

\bibitem[Klopp et~al., 2017]{klopp2017robust}
Klopp, O., Lounici, K., and Tsybakov, A.~B. (2017).
\newblock Robust matrix completion.
\newblock {\em Probability Theory and Related Fields}, 169(1-2):523--564.

\bibitem[Klopp et~al., 2019]{klopp2017structured}
Klopp, O., Lu, Y., Tsybakov, A.~B., and Zhou, H.~H. (2019).
\newblock Structured matrix estimation and completion.
\newblock {\em Bernoulli}, 25(4B):3883--3911.

\bibitem[Koltchinskii et~al., 2011]{KLT11}
Koltchinskii, V., Lounici, K., and Tsybakov, A.~B. (2011).
\newblock Nuclear-norm penalization and optimal rates for noisy low-rank matrix
  completion.
\newblock {\em The Annals of Statistics}, 39(5):2302--2329.

\bibitem[Koop and Potter, 2004]{koop2004forecasting}
Koop, G. and Potter, S. (2004).
\newblock Forecasting in dynamic factor models using bayesian model averaging.
\newblock {\em The Econometrics Journal}, 7(2):550--565.

\bibitem[Kumar and Patel, 2010]{kumar2010using}
Kumar, M. and Patel, N.~R. (2010).
\newblock Using clustering to improve sales forecasts in retail merchandising.
\newblock {\em Annals of Operations Research}, 174(1):33--46.

\bibitem[Kuznetsov and Mohri, 2015]{KM15}
Kuznetsov, V. and Mohri, M. (2015).
\newblock Learning theory and algorithms for forecasting non-stationary time
  series.
\newblock In {\em Advances in neural information processing systems}, pages
  541--549.

\bibitem[Lam and Yao, 2012]{lam2012factor}
Lam, C. and Yao, Q. (2012).
\newblock Factor modeling for high-dimensional time series: inference for the
  number of factors.
\newblock {\em The Annals of Statistics}, 40(2):694--726.

\bibitem[Lam et~al., 2011]{lam2011estimation}
Lam, C., Yao, Q., and Bathia, N. (2011).
\newblock Estimation of latent factors for high-dimensional time series.
\newblock {\em Biometrika}, 98(4):901--918.

\bibitem[Lippi et~al., 2013]{lippi2013short}
Lippi, M., Bertini, M., and Frasconi, P. (2013).
\newblock Short-term traffic flow forecasting: An experimental comparison of
  time-series analysis and supervised learning.
\newblock {\em IEEE Transactions on Intelligent Transportation Systems},
  14(2):871--882.

\bibitem[Lipton et~al., 2016]{lipton2016modeling}
Lipton, Z.~C., Kale, D.~C., and Wetzel, R. (2016).
\newblock Modeling missing data in clinical time series with rnns.
\newblock {\em Machine Learning for Healthcare}, 56.

\bibitem[London et~al., 2014]{LHTG14}
London, B., Huang, B., Taskar, B., and Getoor, L. (2014).
\newblock {PAC-Bayesian} collective stability.
\newblock In {\em Artificial Intelligence and Statistics}, pages 585--594.

\bibitem[McDonald et~al., 2017]{MSS17}
McDonald, D.~J., Shalizi, C.~R., and Schervish, M. (2017).
\newblock Nonparametric risk bounds for time-series forecasting.
\newblock {\em Journal of Machine Learning Research}, 18(32):1--40.

\bibitem[Meir, 2000]{Mei00}
Meir, R. (2000).
\newblock Nonparametric time series prediction through adaptive model
  selection.
\newblock {\em Machine learning}, 39(1):5--34.

\bibitem[Moridomi et~al., 2018]{moridomi2018tighter}
Moridomi, K., Hatano, K., and Takimoto, E. (2018).
\newblock Tighter generalization bounds for matrix completion via factorization
  into constrained matrices.
\newblock {\em IEICE Transactions on Information and Systems},
  101(8):1997--2004.

\bibitem[Negahban and Wainwright, 2011]{negahban2011estimation}
Negahban, S. and Wainwright, M.~J. (2011).
\newblock Estimation of (near) low-rank matrices with noise and
  high-dimensional scaling.
\newblock {\em The Annals of Statistics}, 39(2):1069--1097.

\bibitem[Poignard, 2019]{Poignard2019}
Poignard, B. (2019).
\newblock Sparse multivariate {ARCH} models: Finite sample properties.
\newblock {\em https://arxiv.org/pdf/1808.05352v1.pdf}.

\bibitem[Purser et~al., 2009]{purser2009use}
Purser, A., Bergmann, M., Lund{\"a}lv, T., Ontrup, J., and Nattkemper, T.~W.
  (2009).
\newblock Use of machine-learning algorithms for the automated detection of
  cold-water coral habitats: a pilot study.
\newblock {\em Marine Ecology Progress Series}, 397:241--251.

\bibitem[Saha and Sindhwani, 2012]{saha2012learning}
Saha, A. and Sindhwani, V. (2012).
\newblock Learning evolving and emerging topics in social media: a dynamic
  {NMF} approach with temporal regularization.
\newblock In {\em Proceedings of the fifth ACM international conference on Web
  search and data mining}, pages 693--702. ACM.

\bibitem[Shalizi and Kontorovich, 2013]{SK13}
Shalizi, C. and Kontorovich, A. (2013).
\newblock Predictive {PAC} learning and process decompositions.
\newblock In {\em Advances in neural information processing systems}, pages
  1619--1627.

\bibitem[Steinwart and Christmann, 2009]{SC09}
Steinwart, I. and Christmann, A. (2009).
\newblock Fast learning from non-iid observations.
\newblock In {\em Advances in neural information processing systems}, pages
  1768--1776.

\bibitem[Steinwart et~al., 2009]{SHS09}
Steinwart, I., Hush, D., and Scovel, C. (2009).
\newblock Learning from dependent observations.
\newblock {\em Journal of Multivariate Analysis}, 100(1):175--194.

\bibitem[Suzuki, 2015]{suzuki2015convergence}
Suzuki, T. (2015).
\newblock Convergence rate of {B}ayesian tensor estimator and its minimax
  optimality.
\newblock In {\em International Conference on Machine Learning}, pages
  1273--1282.

\bibitem[Vapnik, 1998]{Va98}
Vapnik, V. (1998).
\newblock {\em Statistical learning theory}.
\newblock Wiley, New York.

\end{thebibliography}
\setlength{\itemindent}{-\leftmargin}
\makeatletter\renewcommand{\@biblabel}[1]{}\makeatother

\end{document}